\documentclass[12pt]{amsart}
\usepackage{amsmath,amssymb,amsbsy,amsfonts,amsthm,latexsym, mathtools,amsopn,amstext,amsxtra,euscript,amscd,mathrsfs,color,bm, cite,multirow,tabularx}
\usepackage{graphicx} 
\usepackage{todonotes,dirtytalk}   
\usepackage{url}
\usepackage[colorlinks,linkcolor=blue,anchorcolor=blue,citecolor=blue]{hyperref}
\usepackage{color}
\usepackage{comment}
\allowdisplaybreaks

\begin{document}
\newtheorem{problem}{Problem}
\newtheorem{theorem}{Theorem}
\newtheorem{lemma}[theorem]{Lemma}
\newtheorem{crit}[theorem]{Criterion}
\newtheorem{claim}[theorem]{Claim}
\newtheorem{cor}[theorem]{Corollary}
\newtheorem{prop}[theorem]{Proposition}
\newtheorem{definition}{Definition}
\newtheorem{question}[theorem]{Question}
\newtheorem{Note}[theorem]{Notation}

\theoremstyle{remark} 
\newtheorem*{note}{Note}
\newtheorem*{rem}{Remark}

\def\cA{{\mathcal A}}
\def\cB{{\mathcal B}}
\def\cC{{\mathcal C}}
\def\cD{{\mathcal D}}
\def\cE{{\mathcal E}}
\def\cF{{\mathcal F}}
\def\cG{{\mathcal G}}
\def\cH{{\mathcal H}}
\def\cI{{\mathcal I}}
\def\cJ{{\mathcal J}}
\def\cK{{\mathcal K}}
\def\cL{{\mathcal L}}
\def\cM{{\mathcal M}}
\def\cN{{\mathcal N}}
\def\cO{{\mathcal O}}
\def\cP{{\mathcal P}}
\def\cQ{{\mathcal Q}}
\def\cR{{\mathcal R}}
\def\cS{{\mathcal S}}
\def\cT{{\mathcal T}}
\def\cU{{\mathcal U}}
\def\cV{{\mathcal V}}
\def\cW{{\mathcal W}}
\def\cX{{\mathcal X}}
\def\cY{{\mathcal Y}}
\def\cZ{{\mathcal Z}}

\def\A{{\mathbb A}}
\def\B{{\mathbb B}}
\def\C{{\mathbb C}}
\def\D{{\mathbb D}}
\def\E{{\mathbb E}}
\def\F{{\mathbb F}}
\def\G{{\mathbb G}}
\def\I{{\mathbb I}}
\def\J{{\mathbb J}}
\def\K{{\mathbb K}}
\def\L{{\mathbb L}}
\def\M{{\mathbb M}}
\def\N{{\mathbb N}}
\def\O{{\mathbb O}}
\def\P{{\mathbb P}}
\def\Q{{\mathbb Q}}
\def\R{{\mathbb R}}
\def\S{{\mathbb S}}
\def\T{{\mathbb T}}
\def\U{{\mathbb U}}
\def\V{{\mathbb V}}
\def\W{{\mathbb W}}
\def\X{{\mathbb X}}
\def\Y{{\mathbb Y}}
\def\Z{{\mathbb Z}}

\def\ep{{\mathbf{e}}_p}
\def\eq{{\mathbf{e}}_q}
\def\cal#1{\mathcal{#1}}

\def\scr{\scriptstyle}
\def\\{\cr}
\def\({\left(}
\def\){\right)}
\def\[{\left[}
\def\]{\right]}
\def\<{\langle}
\def\>{\rangle}
\def\fl#1{\left\lfloor#1\right\rfloor}
\def\rf#1{\left\lceil#1\right\rceil}
\def\le{\leqslant}
\def\ge{\geqslant}
\def\eps{\varepsilon}
\def\mand{\qquad\mbox{and}\qquad}

\def\sssum{\mathop{\sum\ \sum\ \sum}}
\def\ssum{\mathop{\sum\, \sum}}
\def\ssumw{\mathop{\sum\qquad \sum}}

\def\vec#1{\mathbf{#1}}
\def\inv#1{\overline{#1}}
\def\num#1{\mathrm{num}(#1)}
\def\dist{\mathrm{dist}}

\def\fA{{\mathfrak A}}
\def\fB{{\mathfrak B}}
\def\fC{{\mathfrak C}}
\def\fU{{\mathfrak U}}
\def\fV{{\mathfrak V}}

\newcommand{\bflambda}{{\boldsymbol{\lambda}}}
\newcommand{\bfxi}{{\boldsymbol{\xi}}}
\newcommand{\bfrho}{{\boldsymbol{\rho}}}
\newcommand{\bfnu}{{\boldsymbol{\nu}}}

\def\GL{\mathrm{GL}}
\def\SL{\mathrm{SL}}

\def\Hba{\overline{\cH}_{a,m}}
\def\Hta{\widetilde{\cH}_{a,m}}
\def\Hb1{\overline{\cH}_{m}}
\def\Ht1{\widetilde{\cH}_{m}}

\def\flp#1{{\left\langle#1\right\rangle}_p}
\def\flm#1{{\left\langle#1\right\rangle}_m}

\def\Zm{\Z/m\Z}

\def\Err{{\mathbf{E}}}
\def\O{\mathcal{O}}

\def\cc#1{\textcolor{red}{#1}}
\newcommand{\commT}[2][]{\todo[#1,color=green!60]{Tim: #2}}
\newcommand{\commB}[2][]{\todo[#1,color=red!60]{Bryce: #2}}

\DeclarePairedDelimiter\ceil{\lceil}{\rceil}
\DeclarePairedDelimiter\floor{\lfloor}{\rfloor}

\newcommand{\comm}[1]{\marginpar{%
\vskip-\baselineskip 
\raggedright\footnotesize
\itshape\hrule\smallskip#1\par\smallskip\hrule}}
\newcolumntype{L}{>{\raggedright\arraybackslash}X}

\def\xxx{\vskip5pt\hrule\vskip5pt}

\def\dmod#1{\,\left(\textnormal{mod }{#1}\right)}


\title[Bounds on the Minkowski constants and a function involving $\varphi$]{Bounds on the Minkowski constants and a function involving $\varphi$}

\author[G. Pelizzari]{Giulia Pelizzari}\thanks{}
\address{School of Science, The University of New South Wales Canberra, Australia}
\email{g.pelizzari@unsw.edu.au}

\author[J. Punch]{James Punch}\thanks{}
\address{School of Science, The University of New South Wales Canberra, Australia}
\email{j.punch@unsw.edu.au}

\pagenumbering{arabic}

\begin{abstract}
In 1887, Minkowski determined the least common multiple of the orders of all finite subgroups of $GL_n(\Q);$ we refer to this number as $M(n)$. In \cite{Katznelson}, Katznelson provides the asymptotic behaviour of $M(n)$, with a small error term. In this paper, we use elementary techniques to find explicit upper and lower bounds on $M(n)$ that improve on Katznelson's results; we also recover his asymptotic result. Our results immediately imply explicit bounds on functions closely related to $M(n)$, which appear in the study of abelian varieties (see, for example, \cite{Silverberg1992}, \cite{MR3720508} and \cite{Ozeki2024}). Finally, we examine the function $\Phi(n)$, which also appears in \cite{Ozeki2024}, defined as the greatest positive integer $m$ for which $\varphi(m)$ divides $2n$. We provide explicit upper bounds on $\Phi$.
\end{abstract}

\maketitle
\let\thefootnote\relax
\footnote{\textit{Affiliation}: School of Science, University of New South Wales, Canberra.}
\footnote{\textit{Corresponding author}: Giulia Pelizzari (g.pelizzari@unsw.edu.au).}

\section{Introduction}

In \cite{Minkowski1887}, Minkowski proves that the least common multiple of the orders of all finite subgroups of $GL_n(\Q)$ is $$M(n):=\prod_q q^{\left\lfloor\frac{n}{q-1}\right\rfloor+\left\lfloor\frac{n}{q(q-1)}\right\rfloor+\left\lfloor\frac{n}{q^2(q-1)}\right\rfloor+\ldots}$$ where the product is over all primes $q\leq n+1$.

In \cite[Theorem 1]{Katznelson}, Katznelson provides an asymptotic bound for $M(n)$:
\begin{prop}\label{prop:Katznelson} For any $\varepsilon>0,$
    $$\log{\left(\frac{M(n)}{n!}\right)}=Kn+O(n^{1/2+\varepsilon}) \text{ as }n\rightarrow\infty,$$ where $$K:=\sum_{q \text{ prime}}\frac{\log{q}}{(q-1)^2} \;\; \text{ and }\;\; 1.22696< K < 1.22697.$$ The sum defining $K$ is taken over all primes.
\end{prop}
We therefore define the error term $E(n)=\log{M(n)}-\log{n!}-Kn$. In Section \ref{sec:2}, we use elementary techniques to give explicit upper and lower bounds for $E(n)$. Corresponding bounds on $M(n)$ will follow immediately.

Several other authors use functions that are closely related to $M(n)$ to study abelian varieties. Explicit bounds on these functions, which we now describe, will follow directly from our results.

In \cite{Silverberg1992}, Silverberg uses Minkowski's method from \cite{Minkowski1887} to find an upper bound for the dimension of a certain extension field $L$. She shows that, given two abelian varieties $A$ and $B$ over the same field $F$, all isogenies from one to the other are defined over a certain extension field $L$. If we define
$$G(n):=\gcd\{\lvert\GL_n(\Z /N\Z)\rvert: N\geq3\} \text{ and } H(n):=\gcd\{\lvert\text{GSp}(n, \Z /N\Z)\rvert: N\geq3\},$$ then we have $[L:F]\mid H(\dim{A})H(\dim{B})$ by \cite[Theorem 4.2]{Silverberg1992}. If $A$ and $B$ are chosen to be the same variety, then $[L:F]\mid H(\dim{A})$ by \cite[Theorem 4.1]{Silverberg1992}.
In \cite[3.1]{Silverberg1992}, Silverberg calculates $G(n)$ and $H(n)$ exactly; in our notation, her results are:
$$G(n)=2^{\lfloor{n/2}\rfloor}M(n)\text{ and } H(n)=\frac{G(2n)}{2^{n-1}}=2M(2n).$$

In \cite{MR3720508}, Guralnick and Kedlaya provide a function, $\tilde{H}(n)$, that can be substituted for $H(n)$ in Silverberg's divisibility bounds in the case where the varieties are the same; we have $\tilde{H}(n)\leq H(n)$. To define $\tilde{H}$, recall that a Fermat prime is a prime of the form $2^{2^k}+1$ for some non-negative integer $k$, and let $r(n,q):=\sum_{i=0}^\infty\left\lfloor\frac{2n}{(q-1)q^i}\right\rfloor$. Then \begin{equation*}
    \tilde{H}(n):= \prod_{q \text{ prime}}q^{r'(n,q)},\text{ where } r'(n,q):= \begin{cases}
        r(n,q)-n-1 & \text{if }q=2,\\
        \max{\{0,r(n,q)-1\}} & \text{if }q\text{ is a Fermat prime},\\
        r(n,q) & \text{otherwise.}
    \end{cases}
\end{equation*}

If an upper bound on $[L:F]$ is all that is needed, one should use bounds due to R\'emond \cite{Remond}. In the case where the varieties are the same and have dimension $s$, he shows that $[L:F]\leq 2\alpha(s)6^{s-1}s!$, where $\alpha(2)=2,\alpha(4)=5,\alpha(6)=7/6$, and $\alpha(s)=1$ otherwise.
This is best possible; for any $s$, the upper bound is achieved for some choice of variety. Additionally, the asymptotic behaviour of this bound is immediately clear. He provides a similar result for the more general case where the varieties are not the same.

In \cite{Ozeki2024}, Ozeki studies the dimension of $p$-primary torsion subgroups of Mordell--Weil groups of CM abelian varieties over certain infinite extensions of $p$-adic fields, obtained by adjoining a $p$-adic field and the Lubin--Tate extension of another $p$-adic field. To do this, he makes use of $H(n)$, and the function $\Phi(n)$, which is defined in the following way: $$\Phi(n) := \max\{m \in \mathbb{Z}_{>0} \;\mid \;\varphi(m) \text{ divides } 2n\}.$$ 
We examine $\Phi(n)$ in Section \ref{sec:3}. Ozeki proves that $\Phi(n)\leq 6n$ for odd $n$; we show that the constant $6$ cannot be improved. Finally, we give explicit upper bounds on $\Phi(n)$ that are valid for even $n$.

Our results on $G, H, \Phi$ and $M$ improve \cite[Theorem 3.2 and Corollary 3.3]{Silverberg1992} and can be applied to \cite[Proposition 4.3]{Silverberg1992} to get an explicit bound. Further applications include estimating the value of $C$ in \cite[Theorems 1.1 and 3.1]{Ozeki2024}, and finding an explicit bound in \cite[Theorem 1.2]{Philip}.

\section{Explicit bounds on \texorpdfstring{$M(n)$ and other related functions}{M(n) and other related functions}}\label{sec:2}
The next theorem gives explicit upper and lower bounds on $E(n)$, which we defined above.

\begin{theorem} \label{thm:explicitG}
    Let $E(n)$ be as above, and assume that $n\geq 2$.
    Then \begin{equation}\label{eq:logGresult}\begin{split}
        &-\frac{1}{2}\sqrt{n}\log{n}-4\sqrt{n}-\frac{7\sqrt{n}}{\log{n}}-\frac{7}{4}\log{n}-K+2\log{2}-\frac{13}{4}-\frac{1}{\log{n}}-\frac{\log{n}}{4\sqrt{n}}-\frac{1}{4\sqrt{n}}
        \\&\leq  E(n)\leq \\&
        \frac{1}{2}\sqrt{n}\log{n}+4\sqrt{n}+\frac{7\sqrt{n}}{\log{n}}+2.8749719-\log{2}+\frac{1}{\log{n}}.
    \end{split}\end{equation}
\end{theorem}

 \begin{rem}
 This result implies Proposition \ref{prop:Katznelson}. For contrast, Katznelson proves that $E(n)<2\sqrt{n}\log{n}+3\sqrt{n}$. He also gives a lower bound, but it involves a sum over primes up to $\sqrt{n}$, and is therefore not easily calculated for large $n$. Figure \ref{fig:logE} shows $E(n)$, our upper and lower bounds, and Katznelson's upper bound, for $2\leq n \leq 1500$.
 \end{rem}

 \begin{figure}
    \centering
\includegraphics[width=1\linewidth]{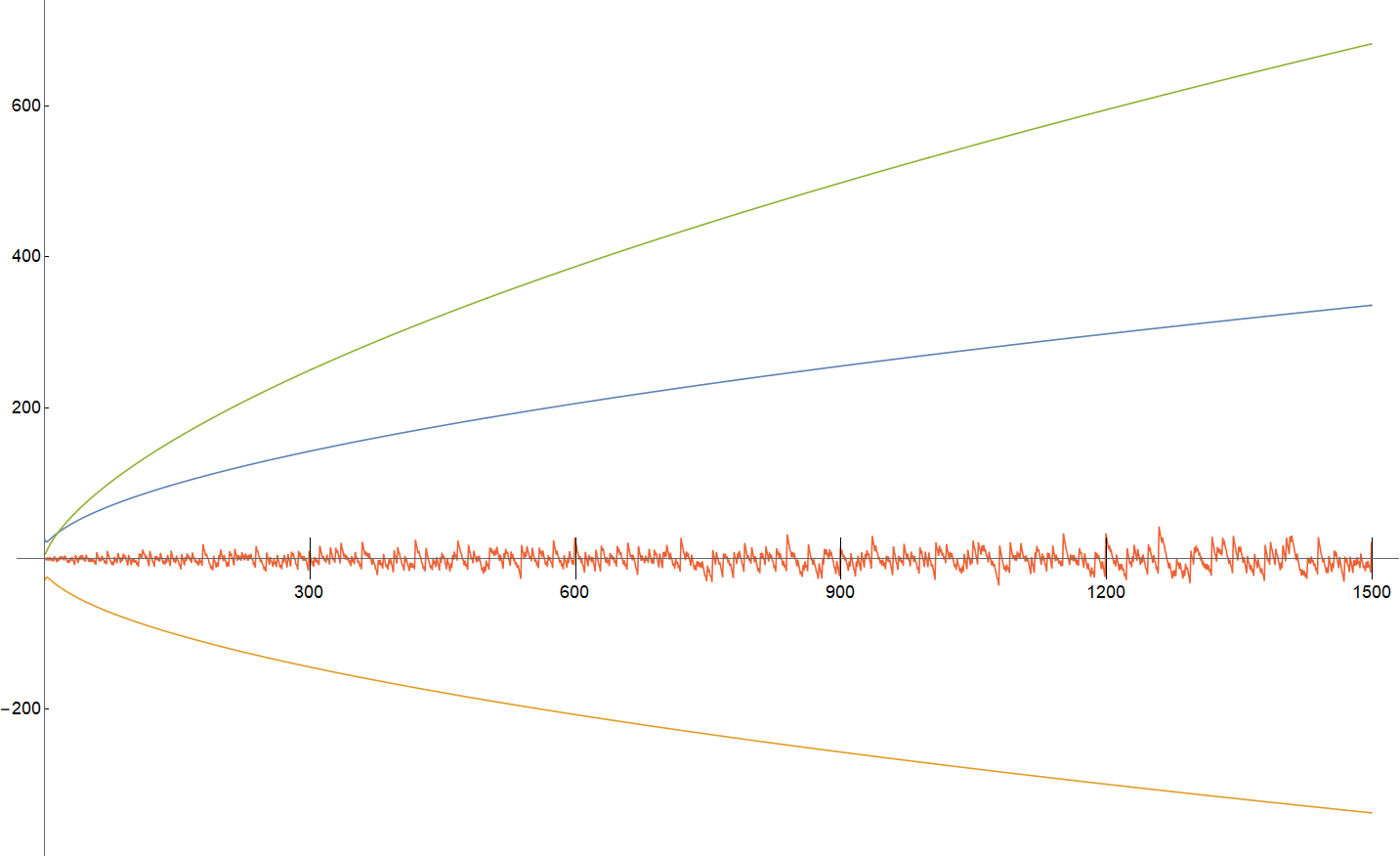}
    \caption{The function $E(n)$ (in red) from $2\leq n \leq 1500$, with upper and lower bounds (in blue and orange respectively) as in Theorem \ref{thm:explicitG}. Katznelson's upper bound is shown in green.}
    \label{fig:logE}
\end{figure}
\begin{proof}
We prove the bounds by assuming that $n\geq 1000$, then verify by computation that they still hold if $2\leq n\leq 999$. We have
    \begin{equation}\label{eq:exactlogG}\begin{split}
        \log M(n) &=\sum_{j=0}^{\infty}\floor*{\frac{n}{2^j}}\log{2}+\sum_{\substack{2<q\leq n \\ q\text{ prime}}}\left(\log{q}\sum_{j=0}^{\infty}\floor*{\frac{n}{q^j(q-1)}}\right)+'\log{(n+1)},\end{split}
    \end{equation} where the $+'$ indicates that we only add the last term if $n+1$ is prime. This is
    \begin{equation*}
    \begin{split}
        \log M(n)&=n\log{2}+\sum_{\substack{2\leq q\leq n \\ q\text{ prime}}}\left(\sum_{j=0}^{\infty}\floor*{\frac{n}{q^{j+1}}}\log{q}\right)\\&+\sum_{\substack{2<q\leq n \\ q\text{ prime}}}\left(\log{q}\sum_{j=0}^{\infty}\left(\floor*{\frac{n}{q^j(q-1)}}-\floor*{\frac{n}{q^{j+1}}}\right)\right)+'\log{(n+1)}.
        \end{split}
    \end{equation*}
    
    Using Legendre's formula for $n!$ (see \cite[Theorem 3.14]{Apostol}), we find that $\log M(n)$ is \begin{equation*}
    \begin{split}
    &n\log{2}+\log{n!}+\sum_{\substack{2<q\leq n \\ q\text{ prime}}}\left(\log{q}\sum_{j=0}^{\infty}\left(\floor*{\frac{n}{q^j(q-1)}}-\floor*{\frac{n}{q^{j+1}}}\right)\right)+'\log{(n+1)}\\
    &=n\log{2}+f_1(n)+\log{n!}+\sum_{\substack{2<q\leq n \\ q\text{ prime}}}\left(\log{q}\sum_{j=0}^{c(n,q)}\left(\floor*{\frac{n}{q^j(q-1)}}-\floor*{\frac{n}{q^{j+1}}}\right)\right),\end{split}\end{equation*}
 where $0\leq f_1(n)\leq \log{(n+1)}$ and $c(n,q)=\floor{\log_q{n}}$. Truncating the sum after $c(n,q)$ is allowed because both floor functions in the sum yield 0 if $q^j>n$.
    


This means that \begin{equation*} \begin{split}
        E(n)&=n\log 2 -Kn + f_1(n)+\\&\sum_{\substack{2<q\leq n \\ q\text{ prime}}}\left(\log{q}\sum_{j=0}^{c(n,q)}\left(\floor*{\frac{n}{q^j(q-1)}}-\floor*{\frac{n}{q^{j+1}}}\right)\right). \end{split}
    \end{equation*}
    
We now examine the sum over primes on the right hand side. Let $\{x\}$ be the fractional part of $x$. \begin{equation*}\begin{split}
        &\sum_{\substack{2<q\leq n \\ q\text{ prime}}}\left(\log{q}\sum_{j=0}^{c(n,q)}\left(\frac{n}{q^j(q-1)}-\left\{{\frac{n}{q^j(q-1)}}\right\}-\frac{n}{q^{j+1}}+\left\{{\frac{n}{q^{j+1}}}\right\}\right)\right)\\
        &=\sum_{\substack{2<q\leq n \\ q\text{ prime}}}\left(\log{q}\sum_{j=0}^{c(n,q)}\left(\frac{n}{q^{j+1}(q-1)}-\left\{{\frac{n}{q^j(q-1)}}\right\}+\left\{{\frac{n}{q^{j+1}}}\right\}\right)\right)=:S_1-S_2, \end{split}
    \end{equation*} where $$S_1=\sum_{\substack{2<q\leq n \\ q\text{ prime}}}\left(\frac{n\log{q}}{q(q-1)}\sum_{j=0}^{c(n,q)}\frac{1}{q^{j}}\right)$$ and $$S_2=\sum_{\substack{2<q\leq n \\ q\text{ prime}}}\left(\log{q}\sum_{j=0}^{c(n,q)}\left(\left\{{\frac{n}{q^j(q-1)}}\right\}-\left\{{\frac{n}{q^{j+1}}}\right\}\right)\right).$$

In $S_1$, since $\log_q{n}\leq c(n,q) \leq 1+\log_q{n}$, the geometric series $S=\sum_{j=0}^{c(n,q)}\frac{1}{q^j}$ satisfies \begin{align*}
        \frac{1-1/n}{1-1/q}=\frac{1-(1/q)^{\log_q n}}{1-1/q}&\leq S \leq \frac{1-(1/q)^{1+\log_q n}}{1-1/q}=\frac{1-1/nq}{1-1/q}
    \end{align*} Therefore 
    \begin{equation}\label{eq:S1}
        (n-1) \sum_{\substack{2<q\leq n \\ q\text{ prime}}}\frac{\log{q}}{(q-1)^2} \leq S_1 \leq n \sum_{\substack{2<q\leq n \\ q\text{ prime}}}\frac{\log{q}}{(q-1)^2} - \sum_{\substack{2<q\leq n \\ q\text{ prime}}}\frac{\log{q}}{q(q-1)^2}.
    \end{equation}
    
    The sum defining $K$ converges absolutely, so we may write the left hand side as
   $$ (n-1)(K-\log{2})  - (n-1) \sum_{\substack{q> n \\ q\text{ prime}}}\frac{\log{q}}{(q-1)^2},$$
 and the sum in the last expression is bounded by
\begin{equation*}
    \sum_{\substack{q> n \\ q\text{ prime}}}\frac{\log{q}}{(q-1)^2} \leq \int_{n}^{\infty} \frac{\log{t}}{(t-1)^2} dt = \left[ - \frac{\log{t}}{t-1}+\log{\frac{t-1}{t}}\right]_{n}^{\infty} = \frac{\log{n}}{n-1}-\log{\left(\frac{n-1}{n}\right)}.
\end{equation*}

This upper bound could be improved, but error terms in other parts of this proof are much larger, so it would be of little consequence. The right hand side of (\ref{eq:S1}) is:
\begin{equation*}
    n(K-\log{2}) - n \sum_{\substack{q> n \\ q\text{ prime}}}\frac{\log{q}}{(q-1)^2} - \left(\sum_{\substack{2<q\leq 1000 \\ q\text{ prime}}}\frac{\log{q}}{q(q-1)^2} +\sum_{\substack{1000<q\leq n \\ q\text{ prime}}}\frac{\log{q}}{q(q-1)^2}\right).
\end{equation*}

We also have \begin{equation*}\begin{split}
    \sum_{\substack{q> n \\ q\text{ prime}}}\frac{\log{q}}{(q-1)^2} &\geq \int_{n+1}^{\infty} \frac{\log{t}}{(t-1)^2}dt = \frac{\log{(n+1)}}{n} - \log{\left(\frac{n}{n+1}\right)}, \\
    \sum_{\substack{2<q\leq 1000 \\ q\text{ prime}}}\frac{\log{q}}{q(q-1)^2} &\geq 0.1250281, \text{ and } \sum_{\substack{1000<q\leq n \\ q\text{ prime}}}\frac{\log{q}}{q(q-1)^2} \geq 0.\end{split}
\end{equation*} Since $- n\log{\left(\frac{n}{n+1}\right)} < 1$ and $(n-1)\log{\left(\frac{n-1}{n}\right)} > -1$, we have
\begin{equation*}
     (K-\log{2})(n-1) - (\log{n}) -1 < S_1 < (K-\log{2})n - \log{(n+1)} +1 - 0.1250281.
\end{equation*}

We now split $S_2$ into two parts: $S_3$, taken over the interval $(2,n^{1/2}+1]$; and $S_4$, taken over $(n^{1/2}+1,n]$. In $S_3$, we bound the fractional parts trivially. Writing $\pi$ for the prime-counting function and $\theta$ for the Chebyshev theta function, \begin{align*}\begin{split}
        \lvert S_3 \rvert &\leq\sum_{\substack{2<q\leq n^{1/2}+1 \\ q\text{ prime}}}\left(\log{q}\sum_{j=0}^{c(n,q)}\left\lvert\left\{{\frac{n}{q^j(q-1)}}\right\}-\left\{{\frac{n}{q^{j+1}}}\right\}\right\rvert\right) \\
        &\leq\sum_{\substack{2<q\leq n^{1/2}+1 \\ q\text{ prime}}}\left(\log{q}\sum_{j=0}^{c(n,q)}1\right)\leq\sum_{\substack{2<q\leq n^{1/2}+1 \\ q\text{ prime}}}(\log{q})(1+\log_q n)\\
        &=\sum_{\substack{2<q\leq n^{1/2}+1 \\ q\text{ prime}}}(\log q+\log n)\\
        &=\theta(n^{1/2}+1)-\log{2}+\log{n}\sum_{\substack{2<q\leq n^{1/2}+1 \\ q\text{ prime}}}1\\
        &\leq \theta(n^{1/2}+1)-\log{2}+(\log{n})\pi(\sqrt{n}).\\
        \end{split}
    \end{align*} By \cite[Theorems 1 and 4]{RosserSchoenfeld1962}, $\pi(n)<\frac{n}{\log{n}}\left(1+\frac{3}{2\log{n}}\right)$ and $\theta(n)<n\left(1+\frac{1}{2\log{n}}\right)$ whenever $n>1$. Since $n\geq 1000$, it follows that \begin{equation*}\begin{split}
        \lvert S_3 \rvert &\leq (n^{1/2}+1)\left(1+\frac{1}{2\log{(n^{1/2}+1)}}\right)-\log{2}+(\log{n})\frac{\sqrt{n}}{\frac{1}{2}\log{n}}\left(1+\frac{3}{\log{n}}\right)\\
        &\leq (n^{1/2}+1)\left(1+\frac{1}{2\log{n^{1/2}}}\right)-\log{2}+2\sqrt{n}\left(1+\frac{3}{\log{n}}\right)\\
        &= 3\sqrt{n}+\frac{7\sqrt{n}}{\log{n}}+1-\log{2}+\frac{1}{\log{n}}.\end{split}
    \end{equation*} Better bounds for $\pi$ and $\theta$ are available, but would not improve the main term.
    
    In $S_4$, observe that if $n\geq q>n^{1/2}+1$, then $\log_q n \geq 1> \frac{1}{2}\log_q n$. That is, $1\leq \log_q n < 2$, so $c(q,n)=1$. Both $\frac{n}{q(q-1)}$ and $\frac{n}{q^2}$ are in the interval $(0,1)$, so each is equal to its fractional part. We therefore have \begin{equation*}\begin{split}
        S_4&=\sum_{\substack{n^{1/2}+1<q\leq n \\ q\text{ prime}}}\left((\log{q})\sum_{j=0}^{1}\left(\left\{{\frac{n}{q^j(q-1)}}\right\}-\left\{{\frac{n}{q^{j+1}}}\right\}\right)\right)\\
        &=\sum_{\substack{n^{1/2}+1<q\leq n \\ q\text{ prime}}}\left(\left\{{\frac{n}{q-1}}\right\}-\left\{{\frac{n}{q}}\right\}+\frac{n}{q(q-1)}-\frac{n}{q^2}\right)\log{q}\\
        &=\sum_{\substack{n^{1/2}+1<q\leq n \\ q\text{ prime}}}\left(\left\{{\frac{n}{q-1}}\right\}-\left\{{\frac{n}{q}}\right\}\right)\log{q}+\sum_{\substack{n^{1/2}+1<q\leq n \\ q\text{ prime}}}\left(\frac{n}{q^2(q-1)}\right)\log{q}=:S_5+S_6.
        \end{split}
    \end{equation*}

    Clearly, $S_6\geq 0$. For an upper bound, observe that if $q>1+\sqrt{n}$, we have $\frac{q}{q-1}=1+\frac{1}{q-1}< 1+\frac{1}{\sqrt{n}}$. Then \begin{equation*}\begin{split}
        S_6&< \sum_{\substack{n^{1/2}+1<q\leq n \\q\text{ prime}}}\frac{n\log{q}}{q^3}\left(1+\frac{1}{\sqrt{n}}\right)\leq (n+\sqrt{n})\int_{n^{1/2}}^{n}\frac{\log{t}}{t^3}\,dt.\end{split}
    \end{equation*}
    
The comparison with the integral holds because $\frac{\log{q}}{q^3}$ is a decreasing function whenever $q>e^{1/3}\approx 1.4$, and we know $q>1+\sqrt{1000}>32$. So, we have $$S_6<(n+\sqrt{n})\left[-\frac{1}{4t^2}-\frac{\log{t}}{2t^2}\right]_{\sqrt{n}}^{n}< \frac{1}{4}\left(1+\frac{1}{\sqrt{n}}\right)(1+\log{n}).$$

Now, \begin{equation*}\begin{split}S_5=\sum_{\substack{n^{1/2}+1<q\leq n \\ q\text{ prime}}}\left(\frac{n}{q-1}-\frac{n}{q}-\left(\floor*{\frac{n}{q-1}}-\floor*{\frac{n}{q}}\right)\right)\log{q},\end{split}\end{equation*} and we have $$\frac{n}{q-1}-\frac{n}{q}=\frac{n}{q(q-1)}\in(0,1)\text{ for any }q\in(n^{1/2}+1,n],$$ so $\floor*{\frac{n}{q-1}}-\floor*{\frac{n}{q}}$ is either zero or one.
    
    Therefore \begin{equation*}\begin{split}
        S_5&=\sum_{\substack{n^{1/2}+1<q\leq n \\ q\text{ prime}}}\left({\frac{n}{q-1}}-\frac{n}{q}\right)\log{q}-\sum_{\substack{n^{1/2}+1<q\leq n \\ q\text{ prime}\\\floor*{\frac{n}{q-1}}-\floor*{\frac{n}{q}}=1}} \log{q}=:S_7-S_8.    \end{split}
    \end{equation*}

For $S_7$, we use the same technique as in $S_6$: $S_7$ is clearly non-negative, and \begin{equation*}\begin{split}
    S_7&=\sum_{\substack{n^{1/2}+1<q\leq n \\ q\text{ prime}}}\frac{n\log{q}}{q^2}\frac{q}{q-1}
    \leq (n+\sqrt{n})\int_{n^{1/2}}^{n}\frac{\log{t}}{t^2}\,dt\\
    &=(n+\sqrt{n})\left[-\frac{1}{t}-\frac{\log{t}}{t}\right]_{\sqrt{n}}^{n}< (n+\sqrt{n})\left(\frac{1}{\sqrt{n}}+\frac{\frac{1}{2}\log{n}}{\sqrt{n}}\right)\\
    &=\frac{1}{2}(1+\sqrt{n})(2+\log{n}).\end{split}
    \end{equation*}

Clearly, $S_8\geq 0$. The condition $\floor*{\frac{n}{q-1}}-\floor*{\frac{n}{q}}=1$ holds if and only if some positive integer $z$ satisfies $q>\frac{n}{z}\geq q-1$. Since $1\leq \frac{n}{q} < n^{1/2}$, we have $z<n^{1/2}$. Therefore \begin{equation*}\begin{split}
        S_8&<\sum_{1\leq z< n^{1/2}} \log{\left(\frac{n}{z}+1\right)}\\&\leq \int_0^{n^{1/2}}\log{\left(\frac{n}{t}+1\right)}\,dt=(\sqrt{n}+n)\log{(1+\sqrt{n})}-\frac{1}{2}n\log{n}\\
        &\leq(\sqrt{n}+n)\left(\log{\sqrt{n}+\frac{1}{\sqrt{n}}}\right)-\frac{1}{2}n\log{n}=\frac{1}{2}\sqrt{n}\log{n}+\sqrt{n}+1.    \end{split}
    \end{equation*}
    
    We have ${E(n)}=-(K-\log{2})n+f_1(n)+S_1-S_3-S_6-S_7+S_8$. Combining the upper and lower bounds on each of these quantities yields (\ref{eq:logGresult}). Finally, computing $E(n)$ for $2\leq n \leq 999$ establishes the correctness of \eqref{eq:logGresult} for these $n$.
    \end{proof}

    Explicit bounds for $M(n)$ follow immediately by observing that $\log{M(n)}=E(n)+Kn+\log{n!}$. 
    Robbins' explicit result for Stirling's formula (see \cite[1]{Robbins1955}) is:
 $$n!=\sqrt{2\pi}n^{n+\frac{1}{2}}e^{r_n-n},\text{ where }\frac{1}{12n+1}<r_n<\frac{1}{12n}.$$
 Taking logarithms, we find that $\log(n!)=n\log n -n + \log{\sqrt{2\pi}}+\frac{1}{2}\log{n}+r_n$.
 
 \begin{cor} Let $n\geq 2$. We have \begin{equation}\label{eq:logMresult}\begin{split}
        &-\frac{1}{2}\sqrt{n}\log{n}-4\sqrt{n}-\frac{7\sqrt{n}}{\log{n}}-\frac{5}{4}\log{n}-K+2\log{2}\\&+\log{\sqrt{2\pi}}-\frac{13}{4}-\frac{1}{\log{n}}-\frac{\log{n}}{4\sqrt{n}}-\frac{1}{4\sqrt{n}}+\frac{1}{12n+1}\\&
        \leq  \log{M(n)} - [n\log{n}+ n(K-1)]\leq 
        \\&
        \frac{1}{2}\sqrt{n}\log{n}+4\sqrt{n}+\frac{7\sqrt{n}}{\log{n}}+\frac{1}{2}\log{n}+2.8749719\\&-\log{2}+\log{\sqrt{2\pi}}+\frac{1}{\log{n}}+\frac{1}{12n}.
    \end{split}\end{equation}
 \end{cor}

 An explicit bound for $\log{G(n)}$, immediately follows from the fact that $\log{G(n)}=\lfloor{n/2}\rfloor \log{2}+\log{M(n)}$ and the estimate $(n/2 -1/2)\log{2}\leq\lfloor{n/2}\rfloor\log{2}\leq(n/2) \log{2}$. Since $\log{H(n)}=\log{2}+\log{M(2n)}$, we have the following corollary.
\begin{cor} Let $n \geq 2$. We have
\begin{equation}\label{eq:logHresult}\begin{split}
        &-\frac{1}{2}\sqrt{2n}\log{2n}-4\sqrt{2n}-\frac{7\sqrt{2n}}{\log{2n}}-\frac{5}{4}\log{2n}+\frac{5}{2}\log{2}+\log{\sqrt{2\pi}}\\&-K-\frac{13}{4}-\frac{1}{\log{2n}}-\frac{\log{2n}}{4\sqrt{2n}}-\frac{1}{4\sqrt{2n}}+\frac{1}{24n+1}
        \\&\leq  \log{H(n)} - \left[2n\log{n} + n \left(2\log{2} -2 + 2K\right)\right]\leq 
        \\& \frac{1}{2}\sqrt{2n}\log{2n}+4\sqrt{2n}+\frac{7\sqrt{2n}}{\log{2n}}+\frac{1}{2}\log{2n}+\log{\sqrt{2\pi}}\\&+2.8749719+\frac{1}{\log{2n}}+\frac{1}{24n}.
        \end{split}
    \end{equation}
    \end{cor}


    

    \begin{cor}\label{cor:niceAsymptoticG} We have
        \begin{align}
        \left\lvert\log{M(n)}-\left[n\log{n}+n\left(K-1\right)\right]\right\rvert&\leq \sqrt{n}\log{n} \text{ if }n\geq 3, \label{eq:logMbound}
                \\
                \left\lvert\log{G(n)}-\left[n\log{n}+n\left(\frac{1}
                {2}\log{2}+K-1\right)\right]\right\rvert&\leq\sqrt{n}\log{n} \text{ if }n\geq 3,\label{eq:logGbound} \\ 
                \left\lvert\log{H(n)}-\left[2n\log{n}+n(2\log{2}+2K-2)\right]\right\rvert&\leq \sqrt{2n}\log{2n}  \text{ if }n\geq 3.\label{eq:logHbound}\\
                M(n)&\leq (\sqrt{6}n)^n\label{eq:Mbound}  \text{ if }n\geq 1,\\
                G(n)\leq (2\sqrt{3}n)^n\text{ and }H(n)&\leq 2\left(2\sqrt{6}n\right)^{2n}\label{eq:GHbound}  \text{ if }n\geq 1.
        \end{align}
    \end{cor}
    \begin{table}[h!]
        \centering
        \begin{tabular}{|c|c|c|c|}
        \hline
            $n$ & $M(n)$ & $G(n)$ & $H(n)$\\
            \hline
            1 & 2 & 2 & 48\\
            2 & 24 & 48 & 11520\\
            \hline
        \end{tabular}
        \caption{Some values of $M(n)$, $G(n)$ and $H(n)$}
        \label{tab:MGHtable}
          \vspace{-25 pt}
    \end{table}

    \begin{rem}
    The first three inequalities do not apply when $n<3$; see Table \ref{tab:MGHtable} for the values of $M$, $G$ and $H$ in these cases.

        We get the following asymptotic results:
        \begin{equation}
            \begin{split}
            \log{M(n)}&=n\log{n}+n\left(K-1\right)+O(\sqrt{n}\log{n}),\\
                \log{G(n)}&=n\log{n}+n\left(\frac{1}{2}\log{2}-1+K\right)+O(\sqrt{n}\log{n}),\\ \label{asymptotich}
                \log{H(n)}&=2n\log{n}+n(2\log{2}-2+2K)+O(\sqrt{n}\log{n}).
            \end{split}
        \end{equation}
        Also, \eqref{eq:GHbound} improves on results of Silverberg \cite[Theorem 3.2 and Corollary 3.3]{Silverberg1992}:
        $$\text{If }n\geq 1\text{, then }G(n)<(6.31n)^n\text{ and }H(n)<2(9n)^{2n}.$$
    \end{rem}
        
\begin{proof}
        Let $l_M(n)$ and $u_M(n)$ be the left and right hand sides of \eqref{eq:logMresult} respectively. Write $L_M(n)=l_M(n)-\frac{1}{12n+1}$; we remove the $\frac{1}{12n+1}$ term because it complicates the argument. Importantly, $L_M(n)$ is still a valid lower bound for $\log{M(n)}$. It is easy to show that $\frac{u_M(n)}{\sqrt{n}\log{n}}$ decreases for $n> 1$ and $\frac{u_M(n)}{n}$ decreases for $n\geq 8$. Similarly, $\frac{L_M(n)}{\sqrt{n}\log{n}}$ is increasing for $n>1$.

        If $n\geq 15917$, then $\frac{u_M(n)}{\sqrt{n}\log{n}}<\frac{u_M(15917)}{\sqrt{15917}\log{15917}}<1$, and $\frac{L_M(n)}{\sqrt{n}\log{n}}>\frac{L_M(15917)}{\sqrt{15917}\log{15917}}>-1$. Computing $M(n)$ for $3\leq n \leq 15916$ establishes \eqref{eq:logMbound}. The proof of \eqref{eq:logGbound} proceeds similarly; we prove the result analytically for $n\geq15988$, and the result follows by calculating $G(n)$ for $3\leq n \leq 15987$. We prove \eqref{eq:logHbound} in the same way, but we need only calculate $H(n)$ for $3\leq n\leq 7923$.

       If $n\geq 157$, then $\frac{u_M(n)}{n}<\frac{u_M(157)}{157}<0.6687$. Therefore \begin{equation}\label{eq:calculatinglogM}
           \log{M(n)}\leq n\log{n}+ (-1+K)n+0.6687n<n\log{n}+n\log{\sqrt{6}}.\end{equation} Calculating $M(n)$ for $1\leq n \leq 157$ then establishes \eqref{eq:Mbound}.           \eqref{eq:GHbound} follows immediately from \eqref{eq:Mbound} because $G(n)=2^{\floor{n/2}}M(n)\leq (\sqrt{2})^nM(n)$, and $H(n)=2M(2n)$. 
           
           Further, we cannot improve the constant $\sqrt{6}$ in \eqref{eq:Mbound} because we have equality when $n=2$. The same is true of $2\sqrt{3}$ and $2\sqrt{6}$ in \eqref{eq:GHbound}, since we have equality when $n=2$ and $n=1$ respectively. In this sense, \eqref{eq:Mbound} and \eqref{eq:GHbound} are best possible.
    \end{proof}
    
 Though we will not do so, it is possible to write out a similar explicit bound for $\tilde{H}(n)$; any interested reader may use Lemma \ref{lemma:Htilde} to do so. To establish the lemma, we estimate the difference $\mathcal{H}(n):=\log{H}(n)-\log{\tilde{H}}(n)$. Notice that $H(n)=2\prod_{q \text{ prime}}q^{r(n,q)}$, so we have \begin{align*}
        \mathcal{H}(n)& =\log{2}+\sum_{q\text{ prime}}[r(n,q)-r'(n,q)]\log{q} \\
        &=\log{2}+(n+1)\log{2}+\sum_{\substack{q\text{ Fermat}\\\text{prime}}}[r(n,q)-\max{\{0,r(n,q)-1\}}]\log{q} \\
        &=(n+2)\log{2}+ \sum_{\substack{q\text{ Fermat}\\\text{prime}}} \min{\{r(n,q),1\}}(\log{q})=:(n+2)\log{2}+ S.
 \end{align*}
 It is trivial that $S\geq 0$. We also have
 \begin{equation*}\begin{split}
     S &\leq \sum_{\substack{q\leq n+1 \\q\text{ Fermat prime}}}\log{q} \leq \sum_{\substack{m\leq n+1 \\m\text{ Fermat number}}}\log{m} =\sum_{\substack{2^{2^k}\leq n \\k\geq0}}\log{(2^{2^k}+1)} \\
     &\leq \sum_{k\leq \log_2{\log_2{n}}}\log{2^{2^k}}+\frac{1}{2^{2^k}}=\log{2}\sum_{k\leq \log_2{\log_2{n}}}2^k +\sum_{k\leq \log_2{\log_2{n}}}\frac{1}{2^{2^k}} \\
     &\leq \left(2^{1+\log_2{\log_2{n}}}-1\right)\log{2}+1 =(2\log_2{n}-1)\log{2}+1= 2\log{n}-\log{2}+1,
 \end{split}
 \end{equation*}
 and we have proved the following statement.
\begin{lemma}\label{lemma:Htilde}
    $$(n+2)\log{2} \leq \mathcal{H}(n) \leq (n+1)\log{2} +2\log{n} +1.$$
\end{lemma}

\section{Explicit bounds on \texorpdfstring{$\Phi(n)$}{Φ(n)}}\label{sec:3}

Let $\nu_p(n)$ be the $p$-adic valuation of $n$, and let $\gamma$ be Euler's constant. The following is the second statement in \cite[Proposition 5.2]{Ozeki2024}.
\begin{prop}\label{thm:Ozeki5.2.2}
    Let $t=\nu_2(n)+2$. Let $p_i$ be the $i^\text{th}$ prime number. Then \begin{equation}\label{eq:Proposition4}\Phi(n)\leq 2n\prod_{i=1}^{t}\frac{p_i}{p_i-1}.\end{equation} In particular, if $n$ is odd then $\Phi(n)\leq 6n$.
\end{prop}
If $n=3^k$ for $k\geq 1$, then $\Phi(n)\geq 6n$: we have $$\varphi(6n)=\varphi(2\times 3^{k+1})=2\times3^k\text{, so }\varphi(6n)\mid 2n.$$ Therefore $\Phi(n)=6n$, so if $\Phi(n)\leq Dn$ holds for all odd numbers, then $D\geq 6$.  

There is no constant $D$ such that this inequality holds for all positive integers; we show this by constructing values of $n$ such that $\frac{\Phi(n)}{n}$ is arbitrarily large.

For integers $k\geq 2$, let $k\#$ denote the product of all prime numbers less than or equal to $k$. It follows immediately from the definition of $\Phi$ that $$\Phi\left(\frac{1}{2}\varphi(m)\right)\geq m,$$ for any positive integer $m$. If $n=\frac{1}{2}\varphi(q\#)$ for any prime number $q\geq 2$, we have \begin{equation*}\begin{split}
        \frac{\Phi(n)}{n}&=\frac{\Phi(\frac{1}{2}\varphi(q\#))}{\frac{1}{2}\varphi(q\#)}\geq \frac{2(q\#)}{\varphi(q\#)}= \frac{2(2\cdot 3\cdot \ldots \cdot q)}{(2-1)(3-1)(5-1)\ldots(q-1)}=2\prod_{\substack{p\leq q \\ p \text{ prime}}}\frac{p}{p-1}.\end{split}
\end{equation*}

By \cite[3.28]{RosserSchoenfeld1962}, we have $$\prod_{\substack{p\leq q \\ p \text{ prime}}}\frac{p}{p-1}>e^{\gamma}\log{q}\left(1 - \frac{1}{2 \log^2{q}}\right).$$ The right hand side becomes arbitrarily large as $q\rightarrow \infty$, which proves the claim.
\begin{theorem}\label{thm:explicitPhi}
    Suppose $n$ is even, let $t=\nu_2(n)+2$ and $t'=\log_2(n)+2$. Then $$\Phi(n)\leq
    \begin{cases}
        9.625n, & 16 \nmid n \\
        2ne^{\gamma}\left(1+\frac{1}{(\log{t'}+\log{\log{t'}})^2}\right)(\log{t'}+\log{(\log{t'}+\log{\log{t'}})}), & 16 \mid n.
    \end{cases}$$
\end{theorem}
\begin{proof}
    First, suppose that $t<6$. By Proposition \ref{thm:Ozeki5.2.2}, we have \begin{equation}\label{eq:primeProductBound}\Phi(n)\leq 2n \cdot \frac{2}{1}\cdot\frac{3}{2}\cdot\frac{5}{4}\cdot\frac{7}{6}\cdot\frac{11}{10}=\frac{77}{8}n=9.625n.\end{equation}
    
    Assume that $t\geq 6$, and apply (3.30), (3.12) and (3.13) from \cite{RosserSchoenfeld1962} to \eqref{eq:Proposition4}:
    \begin{align}
        \Phi(n)&< 2n e^{\gamma}\log{p_t}\left(1+\frac{1}{\log^2{p_t}}\right)\notag \\
        &< 2n e^{\gamma}\left(1+\frac{1}{\log^2{(t\log{t})}}\right)\log{\left[t(\log{t}+\log{\log{t}})\right]}\notag \\
        &=2ne^{\gamma}\left(1+\frac{1}{(\log{t}+\log{\log{t}})^2}\right)(\log{t}+\log{(\log{t}+\log{\log{t}})}).\label{eq:phiBound}
    \end{align} The conditions $t\geq 6$, $\nu_2(n)\geq 4$ and $16\mid n$ are equivalent, so it remains to show that \begin{equation}\label{eq:logEquation}\left(1+\frac{1}{(\log{t}+\log{\log{t}})^2}\right)(\log{t}+\log{(\log{t}+\log{\log{t}})})\end{equation} is an increasing function of $t$ when $t\geq 6$. The result will follow because $t' \geq t$.

    Consider the derivative of (\ref{eq:logEquation}); the denominator, $(t\log{t})(\log{t}+\log{\log{t}})^3$, is positive if $t\geq 6$. The numerator is \begin{align*}c(t)&:=1-\log{t}+\log^3{t}+\log^4{t}+3(\log{t})(\log{\log{t}})+2(\log^2{t})(\log{\log{t}})\\&+3(\log^3{t})(\log{\log{t}})+(\log{\log{t}})^2+(\log{t})(\log{\log{t}})^2+3(\log^2{t})(\log{\log{t}})^2\\&+(\log{t})(\log{\log{t}})^3-2(1+\log{t})(\log{(\log{t}+\log{\log{t}})}).\end{align*} All but the second term and last term are positive whenever $t\geq 6$.

    In the last term of $c(t)$, observe that $\frac{1+\log{t}}{\log{t}}$ is maximised when $t=6$ (since $t\geq 6$),  so $1+\log{t}\leq\frac{1+\log{6}}{\log{6}}\log{t}<1.56\log{t}.$ Similarly, $\frac{\log{t}+\log{\log{t}}}{\log{t}}$ is maximised when $t=e^e$, and the maximum value is $1+\frac{1}{e}$. So, the last term is \begin{align*}
        -2(1+\log{t})(\log{(\log{t}+\log{\log{t}})})&>-(3.12\log{t})(\log{(1+1/e)}+\log{\log{t}})\\
        &>-0.98\log{t}-3.12(\log{t})(\log{\log{t}}).
    \end{align*}
    
    Therefore, the sum of the negative terms in $c(t)$ is greater than $-1.98\log{t}-3.12(\log{t})(\log{\log{t}})$.
    Since $\log^3{6}>5$, $\log^4{t}$ (the fourth term of $c(t)$) is greater than $5\log{t}$, so it more than cancels the $-1.98\log{t}$. The sum of the fifth and sixth terms, $3(\log{t})(\log{\log{t}})+2(\log^2{t})(\log{\log{t}})$, is greater than $5(\log{t})(\log{\log{t}})$ since $\log^2{t}>\log{t}$, so it more than cancels the $-3.12(\log{t})(\log{\log{t}})$. It follows that $c(t)$ is positive whenever $t\geq 6$, which completes the proof.
\end{proof}
\begin{rem}
    The function bounding $\Phi(n)$ when $16\mid n$ is asymptotic to $2e^\gamma n\log{\log{n}}$. For any $C>2e^\gamma$, we therefore have $\Phi(n)<Cn\log{\log{n}}$ for sufficiently large $n$ (this also follows from \cite[Proposition 5.3]{Ozeki2024}). We now find $(C,n_0)$ such that $\Phi(n)<Cn\log{\log{n}}$ whenever $n\geq n_0$. Ozeki gives $\left(4,e^{(1.001e)^9}\right)$ in \cite{Ozeki2024}, and we provide $(6.49,25)$.
\end{rem}

\begin{prop}\label{prop:explicitPhiBound}
    If $n\not \in S:=\{1,2,3,4,5,6,8,9,10,12,16,18,20,24\}$, then we have $\Phi(n)<6.49n\log{\log{n}}$.
\end{prop}
\begin{rem}
Table \ref{tab:PhiValues} provides values of $\Phi(n)$ for $n\in S$. Improving upon $6.49$ by even a small amount would increase the size of $S$; indeed, $\frac{\Phi(48)}{48\log{\log{48}}}>6.46.$
\end{rem}
\begin{proof}
Let $t'$ be as in Theorem \ref{thm:explicitPhi}, so that $\log{\log{n}}=\log{((t'-2)\log{2})}$. First, we will assume that $2^{20}\mid n$, so that $t'\geq t\geq 22$. Then
\begin{equation*}
\frac{\Phi(n)}{n\log{\log{n}}}\leq  \frac{2e^{\gamma}\left(1+\frac{1}{(\log{t'}+\log{\log{t'}})^2}\right)(\log{t'}+\log{(\log{t'}+\log{\log{t'}})})}{\log{((t'-2)\log{2})}}.
\end{equation*}

First, $\left(1+\frac{1}{(\log{t'}+\log{\log{t'}})^2}\right)$ is a decreasing function of $t'$, so it is no greater than $$\left(1+\frac{1}{(\log{22}+\log{\log{22}})^2}\right)<1.05617.$$

Then, $\frac{\log{t'}}{\log{(t'-2)+\log{\log{2}}}}$ is decreasing with respect to $t'$ when $t'\geq 10$, so $$\frac{\log{t'}}{\log{(t'-2)+\log{\log{2}}}}<\frac{\log{22}}{\log{20}+\log{\log{2}}}<1.17566.$$

In the proof of Theorem \ref{thm:explicitPhi}, we observed that $(\log{t'}+\log{\log{t'}})\leq \left(1+\frac{1}{e}\right)\log{t'}$, so $$ \log{(\log{t'}+\log{\log{t'}})}< 0.31327+\log{\log{t'}}.$$

Since $\frac{0.31327}{\log{(t'-2)+\log{\log{2}}}}$ is decreasing, its maximum value is $\frac{0.31327}{\log{20}+\log{\log{2}}}<0.11915$.

Assuming for now that $\frac{\log{\log{t'}}}{\log{(t'-2)+\log{\log{2}}}}$ is decreasing for $t'\geq 22$, it is no larger than $\frac{\log{\log{22}}}{\log{20}+\log{\log{2}}}<0.42922$. From all these inequalities, we find that $$\frac{\Phi(n)}{n\log{\log{n}}}<2e^\gamma(1.05617)(1.17566+0.11915+0.42922)<6.49.$$ Now assume that $2^{20}\nmid n$, so $t<22$. The argument of \eqref{eq:primeProductBound} yields $\Phi(n)<15.87n$. And, if $n\geq 102132$, then $15.87n<6.49n\log{\log{n}}$. Computing exact values of $\Phi(n)$ when $n\leq 102131$ establishes the result.

It only remains to prove that $\frac{\log{\log{t'}}}{\log{(t'-2)+\log{\log{2}}}}$ is decreasing when $t'\geq 22$. By differentiating, it suffices to show that \begin{equation*}\begin{split}\frac{1}{t'\log{t'}}-\frac{\log{\log{t'}}}{(t'-2)(\log{(t'-2)}+\log{\log{2}})}&<0,\text{ or equivalently,}\\(t'-2)(\log{(t'-2)}+\log{\log{2}})&<t'(\log{t'})(\log{\log{t'}}).\end{split}\end{equation*} This is implied by \begin{align}t'(\log{(t'-2)}+\log{\log{2}})&<t'(\log{t'})(\log{\log{t'}})\impliedby\notag\\
    0&<(\log{t'})(\log{\log{t'}})-\log{(t'-2)}-\log{\log{2}}\label{eq:tEquation}.
    \end{align} Now, (\ref{eq:tEquation}) holds when $t'=22$. The derivative of its right hand side is $$\frac{1+\log{\log{t'}}}{t'}-\frac{1}{t'-2}.$$ If $t'\geq 22$, then $0<(t'-2)(\log{\log{t'}})-2$, and thus
    \begin{equation*}\begin{split}
        t'<(t'-2)(1+\log{\log{t'}})\implies
        \frac{1}{t'-2}<\frac{1+\log{\log{t'}}}{t'},\end{split}
    \end{equation*} so the right hand side of \eqref{eq:tEquation} is increasing. This completes the proof.
\end{proof}

\begin{table}
    \centering
    \begin{tabular}{|c|c|c|c|c|c|c|c|c|c|c|c|c|c|c|}
    \hline
        $n$ & 1 & 2 & 3 & 4 & 5 & 6 & 8 & 9 & 10 & 12 & 16 & 18 & 20 & 24 \\
    \hline
        $\Phi(n)$ & 6 & 12 & 18 & 30 & 22 & 42 & 60 & 54 & 66 & 90 & 120 & 126 & 150 & 210\\
    \hline
    \end{tabular}
    \caption{Some values of $\Phi(n)$}
    \label{tab:PhiValues}
\end{table}

\section*{Acknowledgements}
The authors are grateful to Tim Trudgian for suggesting the problem, reading the manuscript and providing helpful feedback. We would like to thank G\"ael Rémond for his insightful comments that significantly enriched our knowledge of the possible applications for this paper. The authors are also thankful to Gabriel Dill for drawing our attention to \cite{Remond}, and to Alice Silverberg and Yoshiyasu Ozeki for their comments. The first author thanks Saunak Bhattacharjee for interesting discussions.

\bibliographystyle{plain}
\bibliography{refs}
\end{document}